\documentclass[11pt]{article}
\usepackage{amsmath, amssymb, amsfonts, amsthm, graphicx,color, mathtools, epstopdf, enumitem }
\usepackage{times, comment}
\usepackage{pgfplots}
\usepackage{tikz}
\usepackage{tikz-cd}

\usetikzlibrary{arrows,automata}
\usepackage[margin=1.2in]{geometry}
\newcommand{\nc}{\newcommand}

\nc{\ds}{\displaystyle}
\nc{\e}{\epsilon}

\nc{\dd}{\partial}
\nc{\td}{\dot{\theta}}
\nc{\yd}{\dot{Y}}
\nc{\pt}{p_\theta}
\nc{\ptz}{p_{\theta_0}}
\nc{\pto}{p_{\theta_1}}
\nc{\ptt}{p_{\theta_2}}

\theoremstyle{plain}
\newtheorem{theorem}{Theorem}[section]
\newtheorem{lemma}[theorem]{Lemma}
\newtheorem{prop}[theorem]{Proposition}
\newtheorem{corollary}[theorem]{Corollary}

\newtheorem{innercustomthm}{Theorem}
\newenvironment{customthm}[1]
  {\renewcommand\theinnercustomthm{#1}\innercustomthm}
  {\endinnercustomthm}
  
\theoremstyle{definition}
\newtheorem*{definition}{Definition}

\newtheorem{theorem*}[theorem]{Remark}

\begin{document}

\title{Chaotic dynamics of a bouncing coin}

\author{Ki Yeun Kim\\ 
\texttt{kkim97@illinois.edu}\\
\date{\vspace{-5ex}}
}

\maketitle
\begin{abstract}We study the dynamics of a bouncing coin whose motion is restricted to the two-dimensional plane. Such coin model is equivalent to the system of two equal masses connected by a rigid rod, making elastic collisions with a flat boundary. We first describe the coin system as a point billiard with a scattering boundary. Then we analytically verify that the billiard map acting on the two disjoint sets produces a Smale horseshoe structure. We also prove that any random sequence of coin collisions can be realized by choosing an appropriate initial condition. 
\end{abstract}
\section{Introduction}
A coin toss is often used as a method for choosing randomly between two options. It is also a familiar model of a random outcome when learning probability and statistics. However, the motion of a tossed coin is completely deterministic, meaning that given the set of initial conditions we can calculate the exact outcome using the laws of physics. Various literature on the dynamics of a coin explore these conflicting ideas. 

In 1986, Keller studied an ideal two-dimensional coin: a thin disk whose rotation axis is parallel to the floor. He proved that for a two-dimensional coin tossed and caught in hand before a collision, the probability of getting either heads or tails approaches .5 in the limit of possible initial velocities \cite{keller}. Diaconis, Holmes, and Montgomery showed that a three-dimensional coin, a disk with no restriction on its rotation axis, is slightly biased by .01 to land on the side it was held at the initial position \cite{diaconis}. If we allow the coin to bounce on the floor, the analysis of the qualitative behavior of the coin gets highly nontrivial. Vulovic and Prange carried out computer simulation to prove that a small perturbation of the initial velocities of a two-dimensional bouncing coin can lead to very different outcomes \cite{vulovic}. Strzalko et al. published a similar numerical result for a three-dimensional coin \cite{strzalko}. 

Another motivation for studying a coin comes from the billiards theory. In the 1960s, the billiard problem appeared in the context of statistical mechanics to study a system of spherical gas particles in a box to verify the Boltzmann ergodic hypothesis \cite{sinai}. More recently, non-spherical particles, which have rotational in addition to translational velocities, were studied to build a more realistic gas model \cite{cowan}. In this direction,
it became natural to consider the behavior of \lq\lq{}a billiard ball with two degrees of freedom\rq\rq{} \cite{baryshnikov2}. Before understanding the full dynamics of a two-dimensional billiard object moving in some general planar domains, we study the object interacting with a simple flat boundary. 

We consider an object consisting of two equal masses connected by a weightless rigid rod (Figure 3a). We assume the object moves under the influence of gravity and makes elastic collisions at the flat boundary. Note that this system can be viewed as an equivalent model to a system of a two-dimensional bouncing coin. %Recall that all the previous work on the collision effects of a coin is based on numerical data and lack analytic interpretation. 
We use this coin model since it gives us relatively convenient constants for our analysis. Using a different two-dimensional coin model such as an actual thin disk (the aforementioned Keller model) or a rod will not change the main results in this paper. In that case, we will only need to modify the moment of inertia of the system and a few other related constants.

In this paper, we analytically prove the existence of a chaotic behavior of this two-dimensional coin model. Unlike a realistic situation where a coin eventually rests, our coin does not stop bouncing. Therefore, we cannot discuss the randomness based on its final position. Instead, we study the infinite sequence of coin bounces. We prove that the two-dimensional bouncing coin can produce any infinite sequence of collisions. In this sense, we may consider the two-dimensional coin with collisions a good randomizer. 

%In this paper, we analytically examine some random and chaotic behavior of a two-dimensional coin model with collisions. The coin model we study consists of two equal masses $m_1$ and $m_2$ connected by a weightless rod in the $xy$ plane (Figure 1a). The coin moves in $y\geq0$ under the influence of gravity $g$, and makes elastic collisions at $y=0$. Since there is no energy loss in the system, the coin does not stop and produces a final result of heads or tails. Instead, there will be infinitely many collisions. 
 \begin{theorem}
Consider a two-dimensional coin consisting of two equal masses $m_L$ and $m_R$ connected by a weightless rod in the $XY$ plane (Figure 3a). Suppose the coin moves in $Y> 0$ under the influence of gravity g, and makes elastic collisions at $Y = 0$. If a collision of $m_L$ or $m_R$ is labeled by $L$ or $R$, then any infinite collision sequence of $L$\rq{}s and
$R$\rq{}s can be realized by choosing an appropriate initial condition.
\end{theorem}

%\begin{figure}
%\centering
%\includegraphics*[width=5.2in, trim=0 1.3cm 0 0.04cm ]{Pic_Fellowship-01.eps}
%\caption{The two-dimensional coin model}
%\end{figure}
In order to prove the theorem, we consider the coin system as a point billiard. On the configuration space of the coin, the coin reduces to a mass point moving in a domain with a scattering boundary (Figure 3b). Then we show that a {Smale horseshoe} \cite{smale} is embedded in the billiard map $\tilde f$. Once we have a horseshoe, we will use the fact that the horseshoe map creates one-to-one correspondence between an infinitely long symbolic sequence to a phase point to complete the proof. 

One way to verify that a certain map contains a horseshoe is by showing that the map satisfies the Conley-Moser conditions \cite{moser}, which are a combination of geometric and analytic criteria. If met, these conditions confirm that the stable manifolds and the unstable manifolds of the invariant set transversally intersect, which directly implies the existence of a horseshoe. In Section 4, we use the coin billiard map $\tilde f$ to construct the horizontal and vertical strips satisfying the Conley-Moser conditions. 

Although there are many naturally occuring systems in science and engineering that are proven to contain a horseshoe \cite{levi, holmes}, to the best of our knowledge, all the studied horseshoe maps are defined on one connected domain. For our coin system, we iterate $\tilde f$ on the union of two disjoint rectangles $\widetilde D=\widetilde D_L \cup \widetilde D_R$. The points in $\widetilde D_L$ or $\widetilde D_R$ respectively correspond to collisions of $m_L$ or $m_R$. We show that $\tilde f$ takes each rectangle into a long and thin strip and wraps around the cylinder as illustrated in Figure 1a. From the figure, we see that $\tilde f(\widetilde D) \cap \widetilde D$ are the six horizontal strips, which we denote by $H_s$ where $s \in S=\{L_1, L_2, L_3, R_1, R_2, R_3\}$. If we iterate $\tilde f$ backward, then $\tilde f^{-1}(\widetilde D) \cap \widetilde D$ are the vertical strips $V_s$ where $s \in S$, as shown in Figure 1b. 
%show that there exist six vertical strips $V_i$, $i=\{L_1, L_2, L_3, R_1, R_2, R_3\}$ and corresponding horizontal strips $H_i = f (V_i)$. 

\begin{figure}
\centering
\includegraphics*[width=3.05in ]{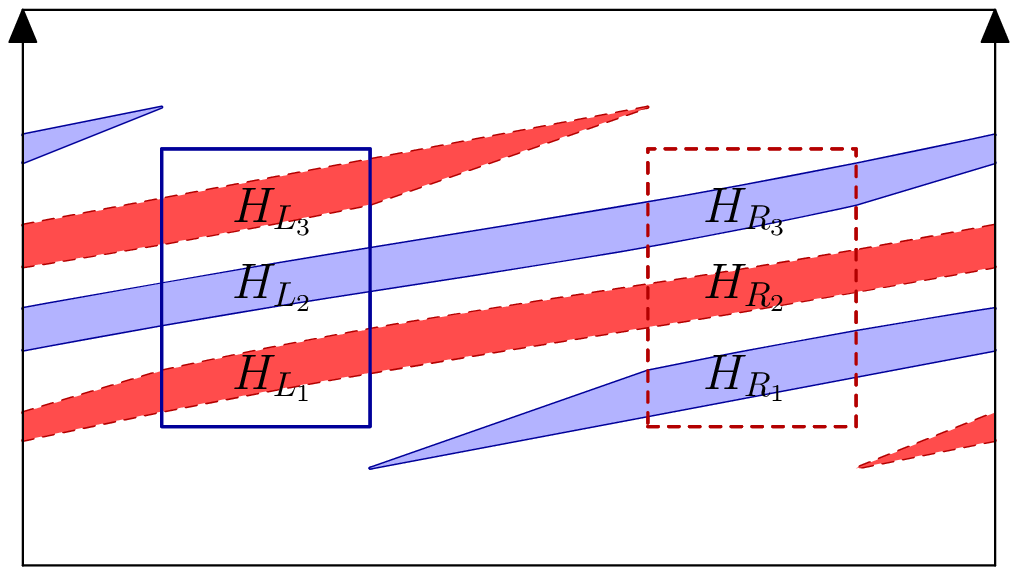}\includegraphics*[width=3.05in ]{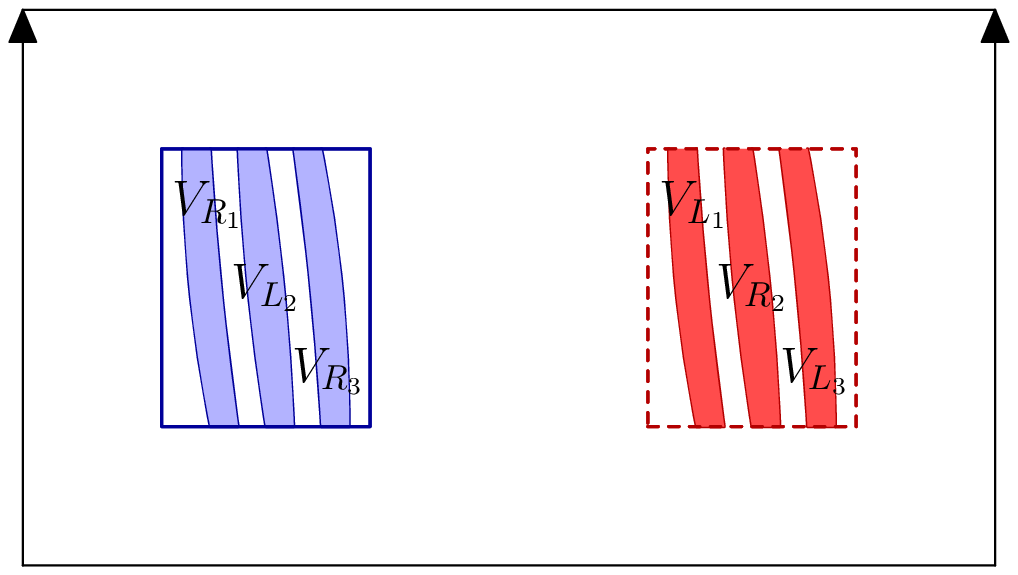}
\caption{(a) The two rectangles $\widetilde D_L$ (solid) and $\widetilde D_R$ (dashed) and their first images under the billiard map $\tilde f$. 
 (b) Vertical strips $V_s$ which are the inverse images of $\widetilde D_L$ and $\widetilde D_R$.}
\end{figure}
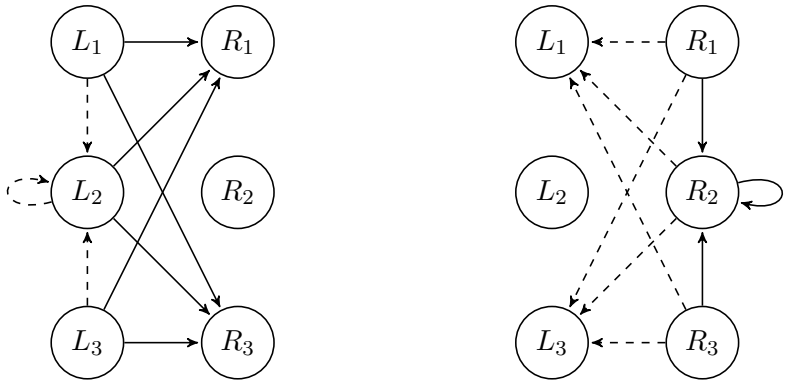
\begin{figure}
\centering
\begin{tikzpicture}[->,>=stealth',shorten >=1pt,auto,node distance=2cm, semithick]
 % \tikzstyle{every state}=[fill=red,draw=none,text=white]

  \node[state] (A)                    {$L_1$};
  \node[state]         (B) [below of=A] {$L_2$};
  \node[state]         (C) [below of=B] {$L_3$};
  \node[state]         (D) [right of=A] {$R_1$};
  \node[state]         (E) [below of=D]       {$R_2$};
  \node[state]         (F) [below of=E]       {$R_3$};

  \path (A) edge             node{} (D)
            edge              node {} (F)
            edge[dashed]              node {} (B)
        (B) edge [loop left, dashed] node {} (B)
            edge              node {} (D)
            edge              node {} (F)
        (C) edge              node {} (F)
            edge[dashed]   node {} (B)
             edge   node {} (D);
\end{tikzpicture} \hspace{3cm}
\begin{tikzpicture}[->,>=stealth',shorten >=1pt,auto,node distance=2cm, semithick]
 % \tikzstyle{every state}=[fill=red,draw=none,text=white]

  \node[state] (A)                    {$L_1$};
  \node[state]         (B) [below of=A] {$L_2$};
  \node[state]         (C) [below of=B] {$L_3$};
  \node[state]         (D) [right of=A] {$R_1$};
  \node[state]         (E) [below of=D]       {$R_2$};
  \node[state]         (F) [below of=E]       {$R_3$};

  \path (D) edge[dashed]             node{} (A)
            edge              node {} (E)
            edge[dashed]              node {} (C)
        (E) edge [loop right] node {} (E)
            edge[dashed]              node {} (A)
            edge [dashed]             node {} (C)
        (F) edge              node {} (E)
            edge [dashed]  node {} (A)
             edge[dashed]   node {} (C);
\end{tikzpicture}

\caption{The valid letter combinations for the sequences in $\Sigma$.}
%\vspace{.3cm}
%\hrule 
\end{figure}

We continue to iterate $\tilde f^n$ for $n \to \pm \infty$ and obtain the invariant Cantor set $\Lambda= \ds \cap_{\tiny{-}\infty}^{\infty} \tilde f^n (\widetilde D)$. To each point in $\Lambda$, we assign a unique bi-infinite sequence $s=(... s_{-2} s_{-1}s_0. s_1 s_2 ...)$ where $s_i \in S$, with some forbidden patterns in $s_i$\rq{}s arising from the overlapping nature of the strips. For example, $\tilde f(H_{L_1}), \tilde f(H_{L_2})$ or $\tilde f(H_{L_3})$ has a nonempty intersection with $H_{R_1}, H_{L_2}$ or $H_{R_3}$, but not with $H_{L_1}, H_{R_2}$ or $H_{L_3}$. We define $\Sigma$ be the set of bi-infinite sequences of six symbols $L_1, L_2, L_3, R_1, R_2, R_3$ with the following rules (Figure 2):\begin{itemize}
\item $L_1, L_2$, or $L_3$ can precede $R_1, L_2$, or $R_3$.  \vspace{-.2cm}
\item $R_1, R_2$, or $R_3$ can precede $L_1, R_2$, or $L_3$.
\end{itemize}
We remark that if we set the representatives $[L]=\{L_1, L_2, L_3\}$ and $[R]=\{R_1, R_2, R_3\}$ and rewrite sequences in $\Sigma$ using $[L]$\rq{}s and $[R]$\rq{}s, then $\Sigma$ contains all possible infinite sequences generated by $[L]$ and $[R]$. It is easy to see that this implies Theorem 1.1. 

The formal horseshoe construction is expressed as:
\begin{theorem}
There is a homeomorphism $\phi: \Lambda \to \Sigma$ such that if we denote the shift map on $\Sigma$ by $\sigma: \Sigma \to \Sigma$, then the diagram below commutes.
\[
\large
\begin{tikzcd}[column sep=large,row sep=large]
\Lambda \arrow[swap]{d}{\phi}\arrow{r}{\tilde f} & \Lambda  \arrow{d}{ \phi} \\
\Sigma \arrow[swap]{r}{\sigma} & \Sigma
\end{tikzcd}
\]
\end{theorem}

%In general, a horseshoe structure is hard to prove without numerical aids. %significance was used by M. Levi (1981) \cite{levi} to complete the picture of limit behavior of all solution and show a horseshoe exists. )

The plan of the paper is the following. We will introduce the coin model, and study its collision dynamics in the language of billiards in Section 2. In Section 3, we recall necessary definitions and state the Conley-Moser conditions. In Section 4, we study the topological picture of the billiard map. The main construction of the horseshoe and the proofs of the main theorems will be presented in Section 5. 
%. 	Due to the Smale–Birkhoff homoclinic theorem (Guckenhermer and Holmes 1983), the existence of transverse homoclinic orbits of a hyperbolic periodic orbit implies the existence of a Smale horseshoe

\section{A two-dimensional coin system}
\begin{figure}[h]
\centering
\includegraphics*[width=5.2in ]{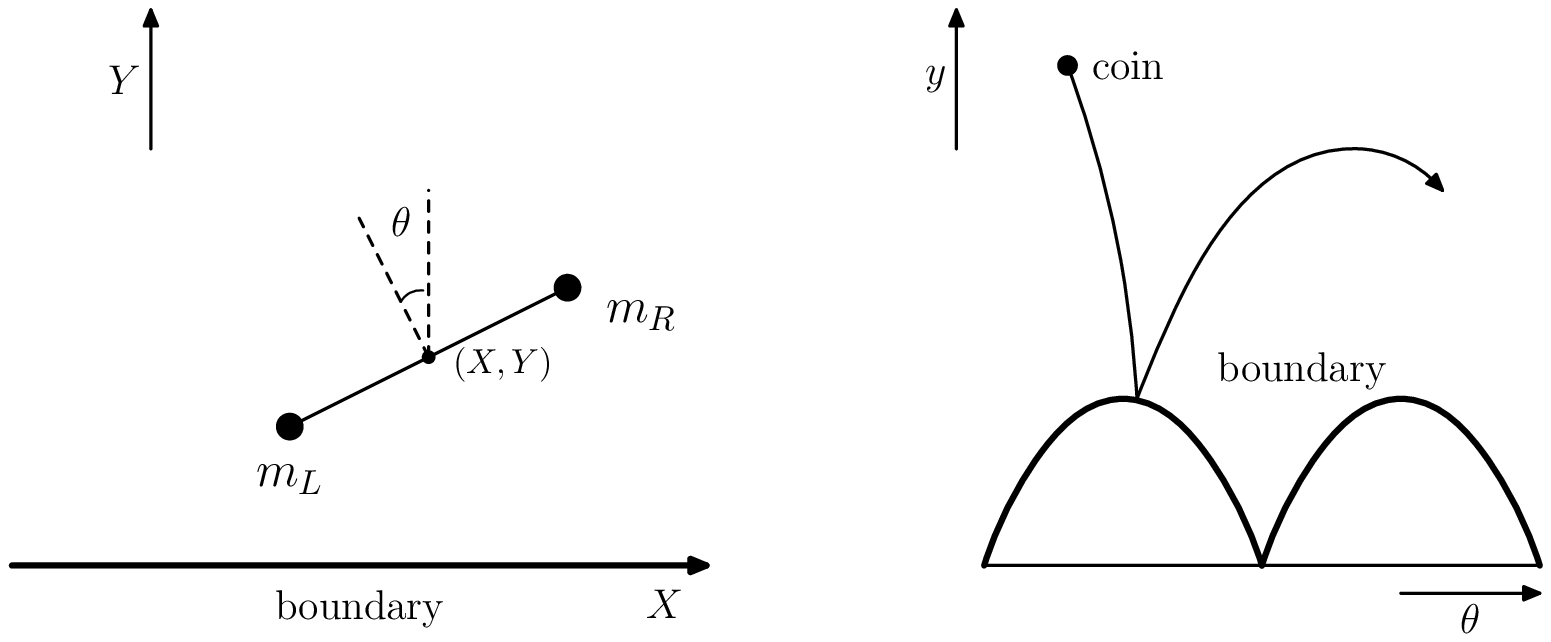}
\caption{(a) The two-dimensional coin model in the $XY$ plane. (b) The coin system as a billiard.}
\end{figure}

The two-dimensional coin model we consider consists of two equal point masses $m_L$ and $m_R$ connected by a weightless rigid rod of length $l$ in the $XY$ plane (Figure 3a). The coordinates of $m_L$, $m_R$, and the center of mass of the coin are denoted by $(X_L, Y_L), (X_R, Y_R)$, and $(X,Y)$, respectively. Let $\theta \in \mathbb{R}$ be the angular position of the normal vector of the coin measured counterclockwise from the positive $Y$-direction.  \begin{theorem*}
Although the coin looks physically the same when $\theta=\theta + 2\pi \mathbb{Z}$, we measure $\theta \in \mathbb{R}$ for now, keeping track of the rotation number. We will identify $\theta$ and $\theta + 2\pi \mathbb{Z}$ as the same in the later section.  
\end{theorem*}
We assume the coin moves under the influence of gravity $g$ in the half space $Y>0$ and reflects elastically at the boundary $Y=0$. The velocity of the center of mass in the $X$-direction is constant since there is no force acting on the system in that direction. Thus, we may assume that the center of mass does not move in the $X$-direction. With this reduction, the configuration space of the coin is $\mathbb{R}^2$ with the $(\theta,Y)$ coordinates. 

On the $\theta Y$ plane, the coin is a mass point in a transformed domain. Note that the coin hits the floor when $Y_L$ or $Y_R$ becomes zero. Therefore, using the relations
\begin{align*}
Y_L=Y- \frac{l}{2}\sin \theta \ge 0  \hspace{1in} Y_R=Y+ \frac{l}{2} \sin \theta \ge 0,
\end{align*}
we find the domain on the $\theta Y$ plane $\{(\theta, Y) : Y\ge \frac{l}{2}|\sin \theta|  \textrm{ for } \theta \in \mathbb{R}\}$. 

Now we examine the motion of the coin. Let $m=m_L+m_R$ be total mass of the coin. The moment of inertia of the coin is given by $I=m_L (\frac{l}{2})^2+ m_R (\frac{l}{2})^2=\frac{1}{4}ml^2$. It is convenient to rescale $Y=\sqrt{\frac{I}{m}} y$ and write the energy of the system as 
\begin{align}
E=\frac{I\dot{\theta}^2}{2} + \frac{m\dot{Y}^2}{2} + mgY= \frac{I\dot{\theta}^2}{2} + \frac{I\dot{y}^2}{2} + g\sqrt{mI}y,
\end{align}
since in this way we can use the standard mirror-like reflection law. By the Hamilton\rq{}s principle of least action, we know that the motion of the coin between two collisions on the $\theta y$ plane is governed by
 \begin{align}\ddot{\theta}=0 \hspace{1in} \ddot{y}=-g\sqrt{\frac{m}{I}}.
 \end{align} 
That is, the coin moves along a parabola between two collisions. We rescale the billiard domain accordingly and get 
\begin{align}
\mathcal Q = \{(\theta, y) : y\ge |\sin \theta|  \textrm{ for } \theta \in \mathbb{R}\}.
\end{align}

To simplify calculations, we choose $m=1, l=2$, thus $I=1$. The results in this paper stay true with different choices of numbers. As we mentioned in the introduction, we may also use a different two-dimensional coin model and perform the same analysis. Different models will change the moment of inertia of the system and a few other constants only. 

At the boundary $\dd \mathcal{Q}$, there will be a mirror-like reflection. We use the subscript $_-$ to denote the values right before a collision and $_+$ to denote the values right after a collision. Recall that given vectors $\mathbf v_-$ and $\mathbf n$, the reflection of  $\mathbf v_-$ across  $\mathbf n$ is given by 
\begin{equation*}
\mathbf{v}_{+} = -2 \frac{\mathbf v_-\cdot \mathbf n }{ \mathbf n \cdot \mathbf n} \mathbf n + \mathbf v_-.
\end{equation*} 
In our case, $\mathbf n$ is the normal vector to the boundary $\dd \mathcal Q$ when $\theta \neq n\pi$, and $\mathbf v_-$ is the incoming velocity
\begin{center}
$\mathbf n = \Big[-\cos \theta \frac{\sin\theta}{|\sin \theta|}, 1\Big]$ \hspace{1cm}
$\mathbf v_-= [\dot \theta_-, \dot y_-].$
\end{center}
Then the reflection law is
\begin{align}
\left(\begin{array}{c}
\dot{\theta}_{+}\\
\dot{y}_{+} 
\end{array}
\right) = 
\frac{1}{1+\cos^2 \theta}
\left(\begin{array}{cc}
\sin^2 \theta & 2\cos \theta \frac{\sin \theta}{|\sin \theta|}\\
2\cos \theta \frac{\sin \theta}{|\sin \theta|} & -\sin^2 \theta
\end{array}
\right)
\left(\begin{array}{c}
\dot{\theta}_-\\
\dot{y}_-
\end{array}
\right).
\end{align}

Using (2)-(4), we build a billiard for the coin system on the $\theta y$ plane. In the domain $\mathcal Q$ on the $\theta y $ plane, we represent the motion of the coin as the billiard flow in $\mathcal Q$. The flow is defined by the trajectory of the coin moving in a parabola until it reaches the boundary $\dd \mathcal Q$ and making a mirror-like reflection at $\dd \mathcal Q$. As in classical billiards, it is sufficient to examine the flow only at the moment of collision. We choose to observe the flow right before a collision. Thus, when there is no ambiguity, we drop the subscript $_{-}$ to simplify our notation. 

On the natural Poincar\'{e} section $\mathcal P$
\begin{align*}\mathcal{P}=\Big\{(\theta, y, \dot{\theta} , \dot{y} ) : (\theta, y) \in \partial \mathcal{Q}, \; \frac{\dot\theta^2}{2} + \frac{\dot y^2}{2} + gy=E\Big\},
\end{align*} we define the return map $f:\mathcal P \to  \mathcal P$, which takes the data of a collision to the data of the next collision.  
 A phase point on $\mathcal P$ can be identified with $(q, v)$ where $q$ is the footpoint on the boundary $\dd \mathcal{Q}$ and $v$ is the incoming velocity vector $v$ at $q$. 
If we restrict our attention to the flow pointing \lq\lq{}downwards\rq\rq{} at the moment of collision,
\[\mathcal P_d= \Big\{(q, v): (q,v) \in \mathcal P, v\cdot {\hat{y}}<0\Big\}, \hspace{.5cm} (\hat y \textrm{ is a vector in the positive } y \textrm{ direction})\] 
then we can use $(\theta, \dot \theta)$ as the coordinates on $\mathcal{P}_d$. Combining (1) and (3), we check that $(q, v)$ on $\mathcal P_d$ can be represented in terms of $(\theta, \dot \theta)$, 
\begin{align}
&q=(\theta, |\sin \theta|) \hspace{1cm} v=\left(\dot \theta, -\sqrt{2E - 2g |\sin \theta| - \dot{\theta}^2}\right).
\end{align}

When a point and its image under $f$ are both on $\mathcal P_d$, we write $f(\theta, \dot \theta)=(\theta_1, \dot \theta_1)$. In general, given $(\theta, \dot \theta)$, we cannot solve for $(\theta_1, \dot \theta_1)$. However, $(\theta_1, \dot \theta_1)$ can be implicitly found in the following way. First, we use (4) to write $(\theta_+, \dot \theta_+)$ in terms of $(\theta, \dot \theta)$, 
\begin{align}
\theta_+ =\theta  \hspace{1cm} \dot \theta_+ = \frac{1}{1+\cos^2 \theta} \left( \dot \theta \sin^2 \theta  - 2\sqrt{2E-2g |\sin \theta| - \dot \theta^2} \cos \theta \frac{\sin \theta}{|\sin \theta|}\right).
\end{align}
Given $(\theta_+, \dot \theta_+)$, we may use the basic laws of physics to determine the formula for the parabolic trajectory $y=\alpha(\theta-\beta)^2 + \gamma$ defined by $(\theta_+, \dot \theta_+)$. We may do the same with $(\theta_1, \dot \theta_1)$. Since $(\theta_+, \dot \theta_+)$ and $(\theta_1, \dot \theta_1)$ should define the same parabolic trajectory, we obtain
\begin{align}
&\theta_+ + \frac{\dot \theta_+ \sqrt{2E -2g |\sin \theta_+| - \dot \theta_+^2}}{g}=\beta= \theta_1 - \frac{\dot \theta_1 \sqrt{2E - 2g |\sin \theta_1| - \dot \theta_1^2}}{g} \\
&|\sin \theta_+| + \frac{2E -2g |\sin \theta_+| - \dot \theta_1^2}{2g}=\gamma=|\sin \theta_1| + \frac{2E - 2g|\sin \theta_1| - \dot \theta_1^2}{2g}. \notag
\end{align}
Combining (6) and (7) lets us implicitly define $(\theta_1, \dot \theta_1)$. 

\section{The Conley-Moser conditions}
In this section, we review the Conley-Moser conditions \cite{moser, wiggins}, which guarantee the existence of a horseshoe. We first need some definitions. Let $D=[X_0, X_1] \times [Y_0, Y_1] \subset \mathbb{R}^2$. 

 \begin{definition}
In $D$, a $\mu_h$-\textit{horizontal curve} is the graph of a function $Y=h(X)$ such that 
\begin{itemize}
\item[(1)] $Y_0\leq h(X) \leq Y_1$, and
\item[(2)] for every pair $X_i, X_j \in [X_0, X_1]$, we have $\ds \left|{h(X_i)-h(X_j)}\right| \leq \mu_h|{X_i-X_j}|$ for some $0\leq \mu_h$.
\end{itemize}
\end{definition}
 \begin{definition}
In $D$, a $\mu_v$-\textit{vertical curve} is the graph of a function $X=v(Y)$ such that 
\begin{itemize}
\item[(1)]$X_0 \leq v(Y) \leq X_1$, and
\item [(2)] for every pair $Y_i, Y_j \in [Y_0, Y_1]$, we have $\ds \left|{v(Y_i)-v(Y_j)}\right| \leq \mu_v|{Y_i-Y_j}|$ for some $0\leq \mu_v$.
\end{itemize}
\end{definition}
\begin{definition}
Given two non-intersecting horizontal curves $h_1(X) < h_2(X)$ in $D$, we define a $\mu_h$-\textit{horizontal strip} as
\begin{align*}
\;\;H&=\left\{(X,Y) : Y\in [h_1(X),h_2(X)] \textrm{ for } X\in \left[X_0, X_1 \right]\right\}.
\end{align*}
Given two non-intersecting vertical curves $v_1(Y) < v_2(Y)$ in $D$, we define a $\mu_v$-\textit{vertical strip} as 
\begin{align*}
V&=\left\{(X,Y) : Y\in [v_1(Y),v_2(Y)] \;\,\textrm{  for } Y\in \left[Y_0, Y_1 \right]\right\}.
\end{align*}
\end{definition}
\begin{definition}
The width of horizontal strips is defined as $d(H)=$ max $|h_2(X)-h_1(X)|$, and the width of vertical strips is defined as $ d(V)=$ max $|v_2(Y)-v_1(Y)|$.
\end{definition}
Consider a diffeomorphism $F:D\to \mathbb{R}^2$ and let $S=\{1,2, \dots N\}$ be an index set. Let $\bigcup_{s\in S} H_s$ be a set of disjoint $\mu_h$-horizontal strips, and let $\bigcup_{s \in S} V_s$ be a set of disjoint $\mu_v$-vertical strips. The \textit{Conley-Moser conditions} on $F$ are:

\begin{itemize}
\item[CM1] $0 \leq \mu_h\mu_v < 1$
\item[CM2]  $F$ maps $V_s$ homeomorphically onto $H_s$. Also, the horizontal (vertical) boundaries of $V_s$ get mapped to the horizontal (vertical) boundaries of $H_s$.
\item[CM3] Suppose $H$ is a $\mu_h$-horizontal strip contained in $\bigcup_{s\in S} H_s$, then $F(H) \cap H_s$ is a $\mu_h$-horizontal strip and $d(F(H) \cap H_s) < d(H)$. Similarly, if $V$ is a $\mu_v$-vertical strip contained in $\bigcup_{s\in S} V_i$, then $F^{-1}(V) \cap V_s$ is a $\mu_v$-vertical strip and $d(F^{-1}(V) \cap V_s) <d(V)$.
\end{itemize}

\begin{theorem}[Moser]
Suppose $F$ satisfies the Conley-Moser conditions. Then $F$ has an invariant Cantor set  $\Lambda$ on which it is topologically conjugate to a full shift on $N$ symbols, i.e. the  diagram below commutes, where $\phi: \Lambda \to \Sigma^N$ is a homeomorphism and $\sigma: \Sigma^N \to \Sigma^N$ is the shift map on the space of sequences of $N$ symbols. 
\[
\large
\begin{tikzcd}[column sep=large,row sep=large]
\Lambda \arrow[swap]{d}{\phi}\arrow{r}{F} & \Lambda  \arrow{d}{ \phi} \\
\Sigma^N \arrow[swap]{r}{\sigma} & \Sigma^N
\end{tikzcd}
\]

\end{theorem}

\begin{proof}
See \cite{moser, wiggins}. 
\end{proof}

\section{The topological picture of $f$}
We will study the topological picture of the return map $f$ which will be fundamental to constructing the horizontal and vertical strips satisfying the Conley-Moser conditions. We start by choosing an appropriate domain for $f$. 
\subsection{Construction of the domain $D$}
 %From (4a)-(6), it is not hard to see that calculating $Df$ requires messy computations. Therefore, 
We naturally want to pick a domain which contains the phase points corresponding to collisions of $m_L$ and also the phase points corresponding to collisions of $m_R$ to study how the two parts interact. From the construction of $\mathcal P$, we see that the phase points $(\theta, \dot \theta)$ for $m_L$ satisfy $2n\pi < \theta < (2n\tiny{+}1) \pi$, and the phase points for $m_R$ satisfy $(2n\tiny{-}1)\pi < \theta < 2n\pi$. Also, note that $(\theta, \dot \theta) \in \mathcal P_d$ is not defined when $\theta = n \pi$, since such points are associated to the billiard trajectories hitting the corners of $\dd\mathcal Q = |\sin \theta|$. Thus, our domain should not include any points with $\theta = n\pi$. To embrace both conditions, it is apparent we must choose two disjoint subsets as a domain. 

Suppose we choose two rectangles of width $2\theta^*\tiny{<}\pi$ and height $2\dot \theta^*$ centered at $\big(\frac{\pi}{2}, 0\big)$ and $\big(\frac{3\pi}{2}, 0\big)$ as the domain. As a first step toward understanding the image of the rectangles, let us attempt to calculate the image of one of the corner points, $\big(\frac{\pi}{2}-\theta^*, \dot \theta^*\big)$. To find its image, we first need to know how $\dot \theta^*$ changes right after the reflection. By (6),
\[\dot \theta_+^* = \frac{1}{1+ \sin^2 \theta^*}\left( (1-\sin^2 \theta^*) \dot \theta^* - 2\sin \theta^* \sqrt{2E - 2g \cos \theta^* - (\dot \theta^*)^2}\right).\]
We can imagine that using the exact form of this equation in the next steps will cause complicated calculations. However, if $\theta^*$ is small and $(\dot \theta^*)^2$ is small compared to $2E$, then we can estimate the $\theta^*$ and $\dot \theta^*$, thus the equation as well, using asymptotics. 

We assume the energy of the system $E$ is large and set $\theta^* = O\big(\frac{1}{E}\big)$ and $\dot \theta^* =O\big(\frac{1}{\sqrt{E}}\big)$. 
We note that, with the bound $\theta^*=O\big(\frac{1}{E}\big)$, we only consider the pieces of the boundary $\dd \mathcal{Q}$ \lq\lq{}near\rq\rq{} the sine peaks. The restriction in $\dot \theta^*$ implies that the angle between the incoming billiard trajectory and $-\hat{y}$ at the moment of a collision is $O\big(\frac{1}{E}\big)$, i.e. the trajectory is \lq\lq{}close to vertical\rq\rq{}. Therefore, this setting guarantees that we will always have the transversal Poincar\`e section.

 Let $D$ be the union of  
\begin{align*} 
{D_L} &= \Big\{ (\theta, \dot\theta): (\theta, \dot \theta) \in {\mathcal P_d}, \Big|\theta-\frac{1}{2}\pi\Big| \leq \frac{k}{E}, |\dot\theta| \leq \frac{\sqrt{2}k}{\sqrt{E}}\Big\}\\
{D_R} &= \Big\{ (\theta, \dot\theta): (\theta, \dot \theta) \in  {\mathcal P_d}, \Big|\theta-\frac{3}{2}\pi\Big| \leq \frac{k}{E}, |\dot\theta| \leq \frac{\sqrt{2}k}{\sqrt{E}}\Big\}, 
\end{align*} where $k$ is a function of $E$ satisfying the equation (8) in Lemma 4.2 below, and $k$ is bounded by two independent constants $0<k_1<k<k_2$. %A phase point in $\widetilde{R_1}$ (or $\widetilde{R_2}$) corresponds to a collision of $m_1$ (or $m_2$) of the coin in the $xy$ plane. 

\begin{theorem*}
We explain why we choose $k(E)$ in this specific way. We claim that the $k$ can be treated as a bifurcation point for topological pictures of $f$. 
\begin{enumerate}
\item If we instead use a constant $K$ slightly smaller than $k$ when defining the domain $D=D(K)$, then $f(D(K))$ intersects $D(K) \tiny {\pm} (2\pi\mathbb{Z}, 0)$ in two horizontal strips and four corners (Figure 4a). 

\item As $K$ approaches $k$, $f(D(K))$ becomes wider and flatter. When we use the exact $k$, then $f(D) \cap (D \tiny {\pm} (2\pi\mathbb{Z}, 0))$ results in six horizontal strips as in Figure 4b. These six horizontal strips are the \lq\lq{}minimal\rq\rq{} number of strips such that, later when we identify the angle of the coin $\theta=\theta + 2\pi \mathbb{Z}$, the image of each rectangle intersects itself \lq\lq{}and\rq\rq{} the other rectangle in full horizontal strips (Figure 1). 

\item With $K$ slightly bigger than $k$, $f(D(K))$ still intersects in $(D(K) \tiny {\pm} (2\pi\mathbb{Z}, 0))$ in six horizontal strips as in the case of $k$, but it is more tedious to compute the coordinates of $f(D(K))$. 

\item As $K$ gets bigger, the topological picture of $f(D(K))$ changes from Figure 4b to Figure 4c. 
\end{enumerate}

\end{theorem*}
\begin{figure}[h]
\centering
\includegraphics*[width=5.5in ]{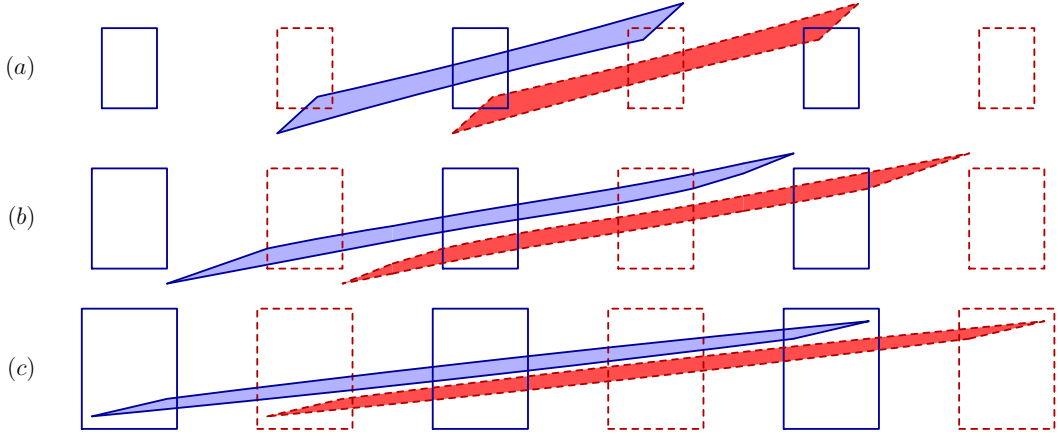}
\caption{(a) When $K<k$. \hspace{.2cm} (b) When using the exact $k$ defined in Lemma 4.2. \hspace{.2cm}(c) When $K>k$.}

\end{figure}
\begin{lemma}
Suppose we have a large enough energy of the system $E$. Then there exists $k=k(E)$ such that given a point $a=(\theta^a, \dot{\theta}^a)= \big(\frac{\pi}{2} - \frac{k}{E},\frac{\sqrt{2}k}{\sqrt{E}}\big) \in \mathcal P_d$, we have

\begin{align}A(E, k)=\frac{2}{g}\left(\dot\theta_+^a \sqrt{2E-2g |\sin \theta^a| -{(\dot{\theta}_+^a)}^2}\right)+\pi=0.
\end{align}
\end{lemma}
\begin{theorem*}
The equation $A(E,k)=0$ implies the following physical meaning: the billiard trajectory on $\mathcal Q$ defined by the initial condition $a=(\theta^a, \dot{\theta}^a)$ hits the boundary again when $\theta = \theta^a-\pi$ (Figure 6). In the phase space, this guarantees that a corner point $a \in D_L$ gets mapped to the left side of $D_R - (2\pi, 0)$ (Figure 5). 
\end{theorem*}
\begin{proof}
Recall from (5) that when $\dot \theta^a=\frac{\sqrt{2}k}{\sqrt{E}}$, we have $\dot y^a= -\sqrt{2E- 2g|\sin \theta^a|-(\dot \theta^a)^2}$. By (6), 
\begin{align} \dot \theta^a_+
&= \frac{1}{1+\sin^2\frac{k}{E}} \left(\left(1-\sin^2 \frac{k}{E}\right) \frac{\sqrt{2}k}{\sqrt{E}}-2\sin \frac{k}{E} \sqrt{2E - 2g\cos\frac{k}{E} - \Big(\sqrt{2E}\frac{k}{E}\Big)^2} \right)\notag\\
&= \frac{1}{1+\sin^2\frac{k}{E}} \left(\left(1-O\left(\frac{k^2}{E^2}\right)\right) \frac{\sqrt{2}k}{\sqrt{E}}-\left(\frac{2k}{E}-O\left(\frac{k^3}{E^3}\right)\right)\left( \sqrt{2E} - O\left(\frac{1}{E^{0.5}}\right) \right) \right)\notag\\
&=\frac{1}{1+\sin^2\frac{k}{E}} \left(-\frac{\sqrt{2}k}{\sqrt{E}}+O\left(\frac{k}{E^{1.5}}\right)\right)=-\frac{\sqrt{2}k}{\sqrt{E}} + O\left(\frac{k}{E^{1.5}}\right).
\end{align}
Then we substitute (9) to (8) 
\begin{align*}
A(E, k)&=\frac{2}{g}\left(\left(-\frac{\sqrt{2}k}{\sqrt{E}}+O\left(\frac{k}{E^{1.5}}\right)\right) \sqrt{2E-2g \cos \frac{k}{E} -\left(-\frac{\sqrt{2}k}{\sqrt{E}}+O\left(\frac{k}{E^{1.5}}\right)\right)^2} \right)+\pi\\
&=\frac{2}{g}\left(\left(-\frac{\sqrt{2}k}{\sqrt{E}}+O\left(\frac{k}{E^{1.5}}\right)\right) \left(\sqrt{2E}-O\left(\frac{1}{E^{1.5}}\right)\right) \right)+\pi\\
&=\frac{2}{g} \left(-2k+O\left(\frac{k}{E}\right)\right)+\pi.
\end{align*}
Since the term $O\big(\frac{k}{E}\big)$ is small, we may assume $A\left(E, \frac{g \pi}{2}\right)<0$ and $A\left(E, \frac{g \pi}{8}\right)>0$. 
By the intermediate value theorem, we can find $k \in \big(\frac{g \pi}{8}, \frac{g \pi}{2}\big)$ such that $A(E, k)=0$. Since the leading term of $\frac{\dd A(E, k)}{\dd k}=-\frac{2}{g}$, it follows that $\frac{\dd A(E, k)}{\dd k} \neq 0$. From the implicit function theorem, we know that we can write $k=k(E)$ in terms of $E$. 
\end{proof}

\subsection{The image of $D$ under the return map $f$}

We will use a few reference points in $D$ and the Jacobian approximation of $f$ to show that $f(D)$ are long and thin strips overlapping $D$ and $D\tiny{\pm}(2\pi, 0)$ as in Figure 5. Due to the symmetry of the coin, $f(D_R)$ will have the same topological picture as $f(D_L)$, shifted by $\pi$. Therefore, we may focus on the topological behavior of $f(D_L)$. We start by analyzing the image of the top edge of $D_L$ under $f$.

\begin{lemma} The image of the top edge of $D_L$ under $f$ is a $\mu_h$-horizontal curve with $\mu_h=O\big(\frac{1}{\sqrt{E}}\big)$.
\end{lemma}
\begin{proof}
We first compute the Jacobian $\mathbf J_f$ for general $(\theta, \dot \theta)$, then estimate the transformation of a vector $(1,0)^T$ in $D_L$ by $\mathbf J_f$. %It is clear that $\mu_h=\ds \max_{(\theta, \dot \theta) \in \textrm{top edge}} \textstyle \Big|\frac{\dd \dot \theta_1}{\dd \theta_{\;}}\big/ \frac{\dd \theta_1}{\dd \theta_{\;}}\Big|$. 
From (6), we get the derivatives
\begin{align*}
\frac{\dd \theta_+}{\dd \theta_{\;\; \;}}&= 1 \hspace{1cm} \frac{\dd \dot \theta_+}{\dd \theta_{\;\;\;}}=\frac{-g\cos^2 \theta ( 3+\cos 2\theta) + 2\dot \theta \dot y \sin 2\theta- 2 \dot y^2 \sin^3\theta}{\dot y(1+\cos^2 \theta)^2}\notag\\ 
\frac{\dd \theta_+}{\dd \dot \theta_{\;\;\;}}&=0 \hspace{1cm} \frac{\dd \dot \theta_+}{\dd \dot \theta_{\;\;\;}}=\frac{1- \cos^2 \theta - (2\dot \theta \cos \theta)/\dot y}{1+ \cos^2 \theta}.
\end{align*}
By implicitly differentiating (7) and simplifying the results with 
\begin{align*}
\dot y=-\sqrt{2E\tiny{-}2g \sin \theta\tiny{-}\dot{\theta}^2} \hspace{.9cm} 
\dot y_+=\sqrt{2E\tiny{-}2g \sin \theta\tiny{-}\dot{\theta}_+^2}\hspace{.9cm}
\dot y_{1}=-\sqrt{2E\tiny{-}2g |\sin \theta_1|\tiny{-}(\dot{\theta}_1)^2},
\end{align*}we obtain 
\begin{align*}
\arraycolsep=-0pt\def\arraystretch{1.7}
\begin{array}{lrcl}
\ds\left(-1\pm \frac{ \dot \theta_1 \cos \theta_1}{\dot y_1}\right)\ds \frac{\dd \theta_1}{\dd \theta_+}-&
\ds \left(\frac{(\dot y_1)^2-(\dot \theta_1)^2}{g \dot y_1}\right)&\ds\frac{\dd \dot \theta_1}{\dd \theta_+}=&\;\ds
\frac{ \dot \theta_+ \cos \theta_+}{\dot y_+}-1\\
&\ds\frac{\theta_1}{g} \;\;&\ds \frac{\dd \dot \theta_1}{\dd \theta_+}=&\;\;0\\
\ds\left(-1\pm \frac{ \dot \theta_1 \cos \theta_1}{\dot y_1}\right)\ds\frac{\dd \theta_1}{\dd \dot \theta_+} -&\ds\left(\frac{(\dot y_1)^2-(\dot \theta_1)^2}{g \dot y_1}\right)&\ds\frac{\dd \dot \theta_1}{\dd \dot \theta_+}=&\;\ds-\frac{(\dot y_+)^2-(\dot \theta_+)^2}{g \dot y_+}\\
&\ds-\frac{\dot \theta_1}{g} \;\;&\ds\frac{\dd \dot \theta_1}{\dd \dot \theta_+} =&\;\ds-\frac{\dot \theta_+}{g},
\end{array}
\end{align*}
where the $\pm$ signs depend on the sign of $\frac{\sin \theta_1}{|\sin \theta_1|}$. Keeping in mind that $\dot \theta_+=\dot \theta_1$, we solve for  
\begin{align*}
\frac{\dd \theta_1}{\dd \theta_+}= \frac{\dot y_1 (\dot y_+  - \dot \theta_+ \cos \theta_+)}{\dot y_+(\dot y_1 \mp \dot \theta_1 \cos \theta_1)}\hspace{.9cm}\frac{\dd \dot \theta_1}{\dd \theta_+}=0 \hspace{.9cm}\frac{\dd \theta_1}{\dd \dot \theta_+}=\frac{(\dot y_+ -\dot y_1)(\dot y_+ \dot y_1+ (\dot \theta_1)^2)}{g \dot y_+ ( \dot y_1 \mp \dot \theta_1 \cos \theta_1)}\hspace{.9cm}\frac{\dd \dot \theta_1}{\dd \dot \theta_+}=1.
\end{align*}

We use the chain rule to compute
\begin{align*}
\mathbf J_f&=\begin{bmatrix}
\ds\frac{\dd \theta_1}{\dd \theta_{\;\;}} &  \ds\frac{\dd \theta_1}{\dd \dot \theta_{\;\;}}\vspace{.2cm}\\
\ds\frac{\dd \dot \theta_1}{\dd \theta_{\;\;}}&  \ds\frac{\dd \dot \theta_1}{\dd \dot\theta_{\;\;}}
\end{bmatrix}
=\begin{bmatrix}
\ds\frac{\dd \theta_1}{\dd \theta_+}\frac{\dd \theta_+}{\dd \theta_{\;\;}} + \frac{\dd \theta_1}{\dd \dot \theta_+}\frac{\dd \dot \theta_+}{\dd \theta_{\;\;}} &  \ds\frac{\dd \theta_1}{\dd \theta_+}\ds\frac{\dd \theta_+}{\dd \dot \theta_{\;\;}}+\frac{\dd \theta_1}{\dd \dot \theta_+}\ds\frac{\dd \dot \theta_+}{\dd \dot \theta_{\;\;}} \vspace{.2cm}\\
\ds\frac{\dd \dot \theta_1}{\dd \theta_+}\frac{\dd \theta_+}{\dd \theta_{\;\;}} + \frac{\dd \dot \theta_1}{\dd \dot \theta_+}\frac{\dd \dot \theta_+}{\dd \theta_{\;\;}} &  \ds\frac{\dd \dot \theta_1}{\dd \theta_+}\ds\frac{\dd  \theta_+}{\dd \dot \theta_{\;\;}}+\frac{\dd \dot \theta_1}{\dd \dot \theta_+}\ds\frac{\dd \dot \theta_+}{\dd \dot \theta_{\;\;}}
\end{bmatrix}.
\end{align*}
For $(\theta, \dot \theta)$ in $D_L$, 
\begin{align}
\begin{array}{l l l}
\hspace{.75cm}\dot \theta=\pm O\big(\frac{1}{\sqrt{E}}\big) & \hspace{.3cm}\dot y_+=\;\;\,O(\sqrt{E}) &\theta = \frac{\pi}{2}\pm O\big(\frac{1}{E}\big) \\
\dot \theta_+, \dot \theta_1=\pm O\big(\frac{1}{\sqrt{E}}\big) & \dot y, \dot y_1=-O(\sqrt{E}) &\theta_1 = O(1).
\end{array}
\end{align} 
Using (10), we estimate the Jacobian $\mathbf J_f$ for the points in $D_L$
\begin{align}
(\mathbf J_f )|_{D_L}&=\begin{bmatrix}
O(1) \cdot 1 + O(\sqrt{E})O(\sqrt{E}) &  O(1)\cdot 0 + O(\sqrt{E}) \cdot 1\\
0 \cdot 1  +1 \cdot O(\sqrt{E})&  0 \cdot 0 + 1 \cdot O(1)
\end{bmatrix}
=\begin{bmatrix}
O(E) &  O(\sqrt{E})\\
O(\sqrt{E})&  O(1)
\end{bmatrix}.
\end{align}
Any vector $(1,0)^T$ in $D_L$ is mapped to $(\mathbf J_f)|_{D_L}(1,0)^T = (O(E), O(\sqrt{E}))^T$. It is clear that the image of the top edge of $D_L$ is a $\mu_h$-horizontal curve with $\mu_h=O\big(\frac{1}{\sqrt{E}}\big)$. Moreover, since we were careful with the signs of the asymptotic terms, it is a monotonically increasing $\mu_h$- horizontal curve with $\mu_h = O\big(\frac{1}{\sqrt{E}}\big)$. 
\end{proof}
\begin{figure}[h]
\centering \includegraphics*[width=6in ]{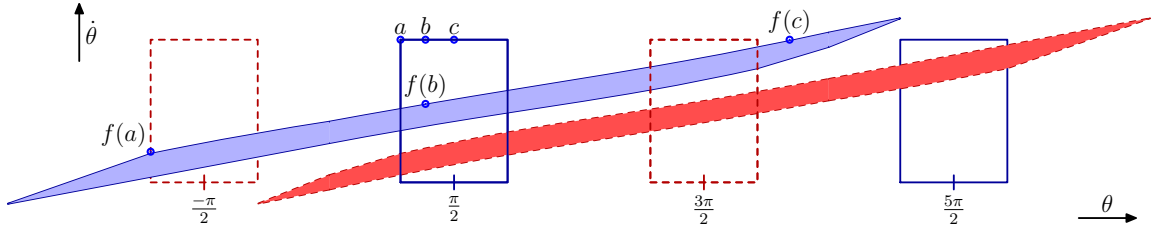}
\caption{A topological picture of $f(D_L \cup D_R)$.}
\end{figure}

Now we consider three reference points $a, b$, and $c$ lying on the top edge of $D_L$ (Figure 5),
\begin{align*}
a=(\theta^a, \dot{\theta}^a)=\Big(\textstyle \frac{\pi}{2} - \frac{k}{E},\frac{\sqrt{2}k}{\sqrt{E}}\Big) \hspace{.7cm} b=(\theta^b, \dot{\theta}^b)=\Big( \theta^b, \frac{\sqrt{2}k}{\sqrt{E}}\Big) \hspace{.7cm} c=(\theta^c, \dot{\theta}^c)=\Big( \frac{\pi}{2}, \frac{\sqrt{2}k}{\sqrt{E}}\Big),
\end{align*}
where ${\theta}^b$ is chosen such that $\dot \theta_+^b=\dot{\theta}_1^b=0$. The existence of such $\theta^b$ can be shown by applying the intermediate value theorem to the function $\dot \theta_1\left(\theta, \sqrt{2E}\frac{k}{E}\right)$. See also Figure 6 for a geometric reason. %From the equation (10) in the proof of Lemma 1, we know $\dot \theta^a_1 = \dot \theta^a_+ = -\sqrt{2E}\frac{k}{E} + O\left(\frac{k}{E^{1.5}}\right)$. Now consider the billiard trajectory determined by $(\theta^c, \dot \theta^c)=(\frac{\pi}{2}, \sqrt{2E}\frac{k}{E})$. Since the trajectory makes a collision with $\dd \mathcal Q$ at the top of its sine peak, $\dot \theta$-value does not change after the collision. Therefore, $\dot \theta^c_1=\sqrt{2E}\frac{k}{E}$.  It follows that there exists $\theta^b \in [\theta^a,\theta^c]$ such that $\dot \theta^b_1 = \dot \theta_1 \left(\theta^b, \sqrt{2E}\frac{k}{E}\right)=0$.

\begin{prop}
The image of $D$ overlaps $D \tiny{\pm}(2\pi \mathbb{Z}, 0)$ in six disjoint $\mu_h$-horizontal strips with $\mu_h=O\big(\frac{1}{\sqrt{E}}\big)$ (Figure 5).
\end{prop}

\begin{proof}

We estimate the images of the three reference points $a,b,$ and $c$ in $D_L$
\begin{align*}
f(a)=(\theta^a_1, \dot \theta^a_1) \hspace{1cm} f(b)=(\theta^b_1, \dot \theta^b_1) \hspace{1cm} f(c)=(\theta^c_1, \dot \theta^c_1).
\end{align*}
From (9), we know $\dot \theta^a_1= \dot \theta_+^a =-\frac{\sqrt{2}k}{\sqrt{E}}+O\big(\frac{k}{E^{1.5}}\big)$. To find $\theta_a^1$, we use the first term of $A(E,k)$ from Lemma 4.2. The term represents twice the $\theta$-distance from $\theta_a$ to the vertex of the parabolic billiard trajectory determined by $(\theta^a, \dot \theta^a)$, and it is set to $-\pi$. In other words, the next collision occurs when \[\theta_a^1 = \theta_a - \pi= -\frac{\pi}{2}-\frac{k}{E}.\]

\begin{figure}[h]
\centering \includegraphics*[width=5in ]{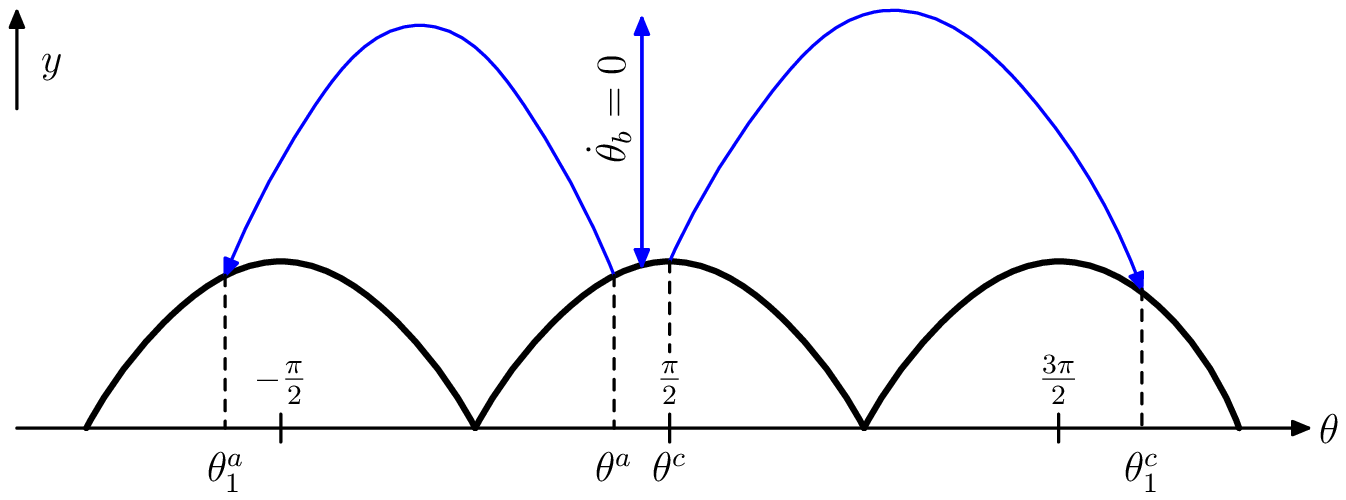}
\caption{The trajectories defined by the initial points $a,b,$ and $c$ on $\mathcal Q$.}

\end{figure}

By construction, \[f(b)=(\theta^b_1, \dot \theta^b_1)=(\theta^b,0).\]
Since the point $c$ is associated to the trajectory hitting the boundary at the peak (Figure 6), the $\dot \theta$-component does not change after the collision, and $\dot \theta^c_1=\sqrt{2E}\frac{k}{E}$. To estimate $\theta^c_1$, we observe that the parabolic billiard trajectory determined by $(\theta^c, \dot \theta^c)$ has the footpoint $q^c=(\theta^c, y^c)=(\frac{\pi}{2}, 1)$ on $\dd \mathcal {Q}$. Any footpoint on $\dd \mathcal {Q}$, including the footpoint for the next collision $q^c_1=(\theta^c_1, y^c_1) $, has the $y$-coordinate less than or equal to $1$. This implies that $y$- and $\theta$-distances from $q^c$ to the vertex of the parabola is smaller than the $y$- and $\theta$-distances from $q^c_1$ to the vertex of the parabola. In other words, 
\begin{align}
\theta^c_1& \geq \theta^c+\frac{2}{g}\left(\dot{\theta}_+^c \sqrt{2E-2g |\sin \theta^c| -(\dot{\theta}_+^c)^2}\right).
\end{align}
We use (9) to get the relation $\dot \theta^c =\frac{\sqrt{2}k}{\sqrt{E}}=|\dot \theta^a |+O\big(\frac{k}{E^{1.5}}\big)$. Then we rewrite (12) as
\begin{align*}
\theta^c_1&\geq\frac{\pi}{2}+\frac{2}{g}\left(\Big(|\dot{\theta}_+^a| + O\Big(\frac{k}{E^{1.5}}\Big)\Big) \sqrt{2E-2g \cos \frac{k}{E} -2g\Big(1-\cos\frac{k}{E}\Big) -\Big(|\dot{\theta}_+^a| + O\Big(\frac{k}{E^{1.5}}\Big)\Big)^2}\right)\notag\\
&=\frac{\pi}{2}+\frac{2}{g}\left(\Big(|\dot{\theta}_+^a| + O\Big(\frac{k}{E^{1.5}}\Big)\Big)\sqrt{2E-2g \sin \theta^a-(\dot{\theta}_+^a)^2  -2g\Big(1-\cos\frac{k}{E}\Big) -  O\Big(\frac{k^2}{E^2}\Big)}\right)\notag\\
&= \frac{\pi}{2}+\frac{2}{g}\left( \Big(|\dot \theta^a |+O\Big(\frac{k}{E^{1.5}}\Big)\Big) \left(\sqrt{2E-2g \sin \theta^a -(\dot \theta_+^a)^2)}- O\Big(\frac{1}{E^{0.5}}\Big)\right)\right)\notag \\
&= \frac{\pi}{2}+\frac{2}{g}\left( |\dot \theta^a |\sqrt{2E-2g \sin \theta^a -(\dot \theta_+^a)^2)}\right)+ \textrm{smaller terms},
\end{align*}
where \begin{align}\textrm{smaller terms}= \frac{2}{g}\left(-|\dot \theta^a |O\Big(\frac{1}{E^{1.5}}\Big)+O\Big(\frac{k}{E^{1.5}}\Big)\sqrt{2E - 2g \sin \theta^a - (\dot \theta_+^a)^2} - O\Big(\frac{k}{E^2}\Big)\right).\end{align}
It can be shown that (13) is $O\big(\frac{k}{E}\big)$ and is greater than or equal to $\frac{k}{E}$. Since we know from (8) that $\frac{2}{g}( |\dot \theta^a |\sqrt{2E-2g \sin \theta^a -(\dot \theta_+^a)^2)})=A(E, k)=\pi$, we conclude that $\theta_c = \frac{3\pi}{2} + \frac{k}{E}+O\big(\frac{k}{E}\big)$.

We plot the images of the three reference points $f(a), f(b),$ and $f(c)$ (Figure 5). Then by Lemma 4.4, we know that the image of the top side of $D_L$ horizontally crosses $D_L, D_R$ and $D_R \tiny{-}(2\pi, 0)$. Similarly, using symmetries, we infer that the image of the bottom edge of $D_L$ also crosses $D_L, D_R$ and $D_R\tiny{-} (2\pi, 0)$ horizontally. Since $f$ is continuous and the proof of Lemma 4.4 can be applied to any vector $(1,0)^T$ lying in $D_L$, we deduce that the image of $D_L$ is $\mu_h$-horizontal strip crossing $D_L, D_R$ and $D_R \tiny{-} (2\pi, 0)$ horizontally. The same argument works for $D_R$, and we conclude that $f(D)$ on $D\tiny{\pm}(2\pi\mathbb{Z}, 0)$ are six horizontal strips. 
\end{proof}

\section{Construction of the strips satisfying the Conley-Moser conditions}
In this section, we finally construct the horizontal and vertical strips satisfying the Conley-Moser conditions using the coin billiard map. Recall that we measured the angular position of the coin $\theta \in \mathbb{R}$ distinguishing $\theta$ from $\theta + 2\pi\mathbb{Z}$. We now impose the equivalence relation $\theta \sim \theta + 2\pi\mathbb{Z}$ and treat $\theta=\theta + 2\pi\mathbb{Z}$ as the same. We denote the resultant objects with $\widetilde{\;\;}$, i.e. $\widetilde{\mathcal{Q}}, \widetilde{\mathcal{P}_d}, \widetilde{D}, \tilde{f}$. 
\subsection{The Conley-Moser conditions for the coin billiard map $\tilde f$.}
Under this setting where $\theta \in S^1$, $\widetilde{D}_L$ and $\widetilde{D}_R$ lie on a cylinder. It directly follows from Proposition 4.5 and Figure 5 that 
\begin{corollary}
$\tilde f(\widetilde D) \cap \widetilde D$ are six $\mu_h$-horizontal strips with $\mu_h=O\big(\frac{1}{\sqrt{E}}\big)$ (Figure 1a).
\end{corollary}
We denote the six horizontal strips $H_s$ where $s\in S=\{L_1, L_2, L_3, R_1, R_2, R_3\}$ as illustrated in Figure 1a. 
\begin{prop}
The preimages of $H_s$ are six disjoint $\mu_v$-vertical strips $V_s \in \widetilde D$ with $\mu_v=O\big(\frac{1}{\sqrt{E}}\big)$. 
\end{prop}
\begin{proof}
Without loss of generality, we choose $H_{L_2}$ to find its preimage. %We will find the inverse image of $H_{L_2}$. %Clearly, the inverse function $f^{-1}$ is defined on all points $p \in H_{L_2}$. 
From the construction of the horizontal strips $H_s$, the inverse images of the upper and lower boundary curves of $H_{L_2}$ are segments of the top and the bottom edges of $\widetilde D_L$. To find the inverse image of the vertical boundaries of $H_{L_2}$, we compute  $\mathbf{J}_{\tilde f^{-1}}(0, 1)^T$.  By the inverse function theorem, $\mathbf{J}_{\tilde f^{-1}}= (\mathbf{J}_{\tilde f} )^{-1}$. Since $f$ and $\tilde f$ are locally the same, we may use (11) and the equations in the proof of Lemma 4.4 to compute $\mathbf{J}_{\tilde f^{-1}}$. First, by direct computation, we find the determinant of $\mathbf J_{\tilde f}$ 
\[ |\mathbf J_{\tilde f}|=\frac{\dot y_1 (\dot y_+ - \dot \theta_1 \cos \theta)(-2\dot{\theta}\cos \theta + \dot{y} \sin^2 \theta)}{\dot{y}  \dot{y}_+ (1+ \cos^2 \theta) (\dot{y}_1 \mp \dot{\theta}_1 \cos \theta_1)}. \] 
Using (10), we find $ |\mathbf J_{\tilde f}| = O(1)$ for the points in $\widetilde D_L$. Thus, 
\[(\mathbf{J}_{\tilde f^{-1}} )|_ {H_{L_2}}= \ds\frac{1}{O(1)} \begin{bmatrix} O(1) &-O(\sqrt{E})\\ -O(\sqrt{E}) & O(E)\end{bmatrix},\] and $(\mathbf{J}_{\tilde f^{-1}} )|_ {H_{L_2}}(0,1)^T=(-O(\sqrt{E}), O(E))^T$. We then obtain $\mu_v=O\big(\frac{1}{\sqrt{E}}\big)$.

To summarize, the preimages of the boundaries of $H_{L_2}$ consist of two segments of the top and the bottom edges of $\widetilde D_L$ and two $\mu_v$-vertical curves in $\widetilde D_L$. Since $\tilde f^{-1}$ is a diffeomorphism, the boundaries get mapped to the boundaries and the interior gets mapped to the interior.  Therefore the inverse image of a horizontal strip $H_{L_2}$ is a vertical strip. 

\end{proof}

\begin{prop}
\leavevmode
\begin{enumerate}
\item Consider a $\mu_h$-horizontal strip $H_s$ for $s \in \{L_1, L_2, L_3\}$. If $j\in \{R_1, L_2, R_3\}$, then  $\tilde f(H_s) \cap H_j$ is a $\mu_h^1$-horizontal strip where $\mu_h^1=O\big(\frac{1}{\sqrt{E}}\big)$ and $d(\tilde f(H_s) \cap H_j) < d(H_s)$. If $j \in \{L_1, R_2, L_3\}$, then $\tilde f(H_s) \cap H_j=\emptyset$. %With $\mu_h<\mu_h = O(\frac{1}{\sqrt{E}})$. 

\item Similary, consider a $\mu_v$-vertical strip $V_s$ for $s \in \{R_1, L_2, R_3\}$. If $j \in \{L_1, L_2, L_3\}$, then $\tilde f^{-1}(V_s) \cap V_j$ is a $\mu_v^1$-vertical where $\mu_v^1=O\big(\frac{1}{\sqrt{E}}\big)$ and $d(\tilde f^{-1}(V_s) \cap V_j) < d(V_s)$. If $j \in \{R_1, R_2, R_3\}$, then $\tilde f^{-1}(V_s) \cap V_j=\emptyset$.

\item The statements 1 and 2 with $L$ and $R$ switched are also true. \end{enumerate}
\end{prop}

\begin{proof}
We only prove the statement 1, since the proofs for the statements 2 and 3 are similar. By observing the topological pictures of $\tilde f(\widetilde D)$ (Figure 1a, Figure 5), we can easily see that when $s \in \{L_1, L_2, L_3\}$ and  $j \in \{L_1, R_2, L_3\}$,  we have $\tilde f(H_s) \cap H_j=\emptyset$. If $j \in \{L_1, R_2, L_3\}$, then $\tilde f(H_s) \cap H_j$ is  nonempty and is contained in $H_j$. It is clear $d(\tilde f(H_s) \cap H_j) \leq d(H_s)$. It remains to estimate the maximal slope of the horizontal boundary curves of $\tilde f(H_s) \cap H_j$. From Proposition 4.5, we know that the horizontal boundary curves of $H_s$ has the maximal slope $\mu_h=O\big(\frac{1}{\sqrt{E}}\big)$. Using (11), we get
\[(\mathbf J_{\tilde f})|_D\begin{pmatrix}1\\\mu_h\end{pmatrix} =\begin{pmatrix}O(E) + O(\sqrt{E})O\big(\frac{1}{\sqrt{E}}\big)\\ O(\sqrt{E}) + O(\frac{1}{\sqrt{E}})\end{pmatrix}=\begin{pmatrix} O(E) \\ O(\sqrt{E})\end{pmatrix}.\]
Thus, the maximal slope of the horizontal boundary curves of $f(H_s) \cap H_j$ is $\mu_h^1=\frac{O(\sqrt{E})}{O(E)} = O\big(\frac{1}{\sqrt{E}}\big)$. 

\end{proof}

\subsection{The proofs of Theorem 1.1 and 1.2}
Corollary 5.1, Proposition 5.2, and Proposition 5.3 together imply that $\tilde f$ satisfies the modified Conley-Moser conditions. Then Theorem 1.2* below is merely a version of Theorem 3.1. To state Theorem 1.2*, we need a few definitions.  

\begin{definition}
Let $\Lambda_N=\{p \;| \;\ds \bigcap_{n=-N}^{N} \tilde f^n(\widetilde D)\}$ and $\Lambda =  \ds \lim_{N\to \infty} \Lambda_N$.
\end{definition}
\begin{definition}
Let $\Sigma$ be the set of bi-infinite sequences of six symbols $L_1, L_2, L_3, R_1, R_2, R_3$ with the following rules:\begin{itemize}
\item $L_1, L_2$, or $L_3$ can precede $R_1, L_2$, or $R_3$. \vspace{-.2cm}
\item $R_1, R_2$, or $R_3$ can precede $L_1, R_2$, or $L_3$.
\end{itemize}
\end{definition}
\begin{customthm}{1.2*}
There is a homeomorphism $\phi: \Lambda \to \Sigma$ such that if we denote the shift map on $\Sigma$ as $\sigma: \Sigma \to \Sigma$, then the diagram below commutes.
\[
\large
\begin{tikzcd}[column sep=large,row sep=large]
\Lambda \arrow[swap]{d}{\phi}\arrow{r}{\tilde f} & \Lambda  \arrow{d}{ \phi} \\
\Sigma \arrow[swap]{r}{\sigma} & \Sigma
\end{tikzcd}
\] %In other words, the sequence $r=\phi({\tilde f}(x)) \in \Sigma$ is obtained from the sequence $s=\phi(x) \in \Sigma$ by shifting indices one place: $r_i=s_{i+1}$. 
\end{customthm}
\begin{proof}
The proof closely follows the proof of Theorem 3. We only point out some key ideas here.

We first describe the invariant set $\Lambda$: the set of points which stay in $\widetilde D$ after infinitely many forward and backward iterations of $\tilde f$. From Corollary 5.1 and Figure 1a, we see that after one forward iteration of $\tilde f$, the invariant set $\tilde f(\widetilde D) \cap \widetilde D$ consists of the $2\cdot 3$ horizontal strips $H_s$ where $s \in S$.  From Proposition 5.3 and Figure 1a, we know that $\tilde f$ acting on $H_s$ creates $3$ nested $\mu_h^1$-horizontal strips with $\mu_h^1=O\big(\frac{1}{\sqrt{E}}\big)$ in each $H_s$. In other words, 
\[\tilde f^2(\widetilde D) \cap \tilde f(\widetilde D) \cap \widetilde D=\tilde f(H_s) \cap H_j  \textrm{ where } s,j \in S\] consists of $2 \cdot 3^2$ thinner $\mu_h^1$-horizontal strips. We define each nested strip 
\[\tilde f(H_{s_{-2}}) \cap H_{s_{-1}}=H_{s_{-2} s_{-1}} \textrm{ where }s_i \in S.\] If we iterate $\tilde f^n$, then $2 \cdot 3^n$$\mu_h^n$-horizontal strips with $\mu_h^n=O\big(\frac{1}{\sqrt{E}}\big)$ remain in $\widetilde D$. We inductively define the nested horizontal strips after $n$ iterations, \[\tilde f(H_{s_{-n} ... s_{-2}}) \cap H_{s_{-1}}= H_{s_{-n} ... s_{-1}}.\]  Similarly, from Propositions 5.2 and 5.3, we obtain $2 \cdot 3^n$ $\mu_v^*$-vertical strips with $\mu_v^n=O\big(\frac{1}{\sqrt{E}}\big)$ as the invariant set for $\tilde f^{-n}$. The nested vertical strips are defined \[\tilde f^{-1}(V_{s_1 ... s_n}) \cap V_{s_0}=V_{s_0 ... s_n}.\]  

We see that $\Lambda_N$ is the intersection of $2 \cdot 3^N$ $\mu_h^N$-horizontal strips and $2\cdot 3^N$ $\mu_v^N$-vertical strips. When $N \to \infty$, we see that $\Lambda$ is a Cantor set, the intersection of infinite number of $\mu_h^*$-horizontal curves and $\mu_v^*$-vertical curves. The first Conley-Moser condition $0 < \mu_h \mu_v = O(\frac{1}{\sqrt{E}}\big) O\big(\frac{1}{\sqrt{E}}\big) <1$ guarantees that there exists a unique intersection point of a $\mu_h$-horizontal curve and a $\mu_v$-vertical curve.

The next step is to assign a sequence $s=(... s_{-2} s_{-1} s_0 s_1...) \in \Sigma$ to each point $p \in\Lambda$ based on to which horizontal curve $H_{...s_{-2}s{-1}}$ and vertical curve $V_{s_0 s_1...}$ the point $p$ belongs. It is clear from the construction that this assignment $\phi: \Sigma \to \Lambda$ is one-to-one. Also, from Proposition 5.3, we know that only certain choices of $s,j$ for $f(H_s) \cap H_j=H_{s j}$ are valid. Thus, the sequences $s \in \Sigma$ must combinatorially follow the two rules mentioned on the previous page. The proof can be completed by showing that $\phi$ is a homeomorphism. See \cite{moser, wiggins} for the full details. 
\end{proof}

\begin{customthm}{1.1*}
For the two-dimensional coin model, if a collision of $m_L$ or $m_R$ is labeled by $L$ or $R$ respectively, then any infinite sequence of $L$\rq{}s and $R$\rq{}s can be realized by choosing an appropriate initial condition.
\end{customthm} 
\begin{proof}
The sequence set $\Sigma$ contains all the infinite sequences of $L_1, L_2, L_3, R_1, R_2, R_3$\rq{}s which do not violate the rules. From the rules, we see that if we set the representatives $[L]=\{L_1, L_2, L_3\}$ and $[R]=\{R_1, R_2, R_3\}$, then either $[L]$ or $[R]$ may follow after $[L]$ (see also Figure 2). Similarly, after $[R]$, either $[L]$ or $[R]$ may follow. In other words, $\Sigma$ contains all possible sequences of $[L]$\rq{}s and $[R]$\rq{}s.

By Theorem 1.2*, having all sequences of $[L]$\rq{}s and $[R]$\rq{}s implies that for any collision sequence of $L$\rq{}s and $R$\rq{}s, there exists a phase point of the coin billiard which makes collisions according to the sequence.

\end{proof}

\vfill
\noindent \large \textbf{Acknowledgements}\vspace{.3cm}

\noindent \normalsize The author thanks Vadim Zharnitsky and Yuliy Baryshnikov for many useful suggestions. This work was supported in part by a gift to the Mathematics Department at the University of Illinois from Gene H. Golub.

\newpage
\bibliographystyle{alpha}

\end{document}